\documentclass{article}

\usepackage{amsmath,amsfonts,amssymb,amsthm,tikz}

\usepackage[utf8]{inputenc}
\usepackage{cite}

\newcommand{\norm}{{\mathcal{N}}}
\newcommand{\calO}{{\mathcal{O}}}

\newcommand{\C}{{\mathbb{C}}}
\newcommand{\F}{{\mathbb{F}}}

\newcommand{\Q}{{\mathbb{Q}}}
\newcommand{\Z}{{\mathbb{Z}}}

\newcommand{\height}{\mathrm{h}}

\newcommand{\gera}{{\mathfrak{a}}}
\newcommand{\gerb}{{\mathfrak{b}}}
\newcommand{\gerp}{{\mathfrak{p}}}

\DeclareMathOperator{\logast}{log^\ast\!}
\DeclareMathOperator{\Li}{Li}

\newcommand{\ph}{\varphi}

    \makeatletter
    \let\@fnsymbol\@alph
    \makeatother

\title{Uniform explicit  Stewart's theorem on prime factors of linear recurrences}
\author{Yuri Bilu\footnote{Supported  by the SPARC Project P445 (India) and ANR project JINVARIANT}, \ Sanoli Gun\footnote{Supported  by the SPARC Project P445 (India)} \ and Haojie Hong\footnote{Supported by the China Scholarship Council grant CSC202008310189}}

\date{Version of \today}

\newtheorem{theorem}{Theorem}[section]
\newtheorem{proposition}[theorem]{Proposition}
\newtheorem{lemma}[theorem]{Lemma}
\newtheorem{corollary}[theorem]{Corollary}
\newtheorem{remark}[theorem]{Remark}

\numberwithin{equation}{section}

\makeatletter

\renewcommand*\l@section[2]{%
  \ifnum \c@tocdepth >\z@
    \addpenalty\@secpenalty
    \addvspace{0.2em \@plus\p@}%
    \setlength\@tempdima{1.5em}%
    \begingroup
      \parindent \z@ \rightskip \@pnumwidth
      \parfillskip -\@pnumwidth
      \leavevmode \bfseries
      \advance\leftskip\@tempdima
      \hskip -\leftskip
      #1\nobreak\hfil \nobreak\hb@xt@\@pnumwidth{\hss #2}\par
    \endgroup
  \fi}

\makeatother

\begin{document}

\hfuzz 4.3pt

\maketitle

\begin{flushright}
\textit{To the memory of Andrzej Schinzel}
\bigskip
\end{flushright}

\begin{abstract}
Stewart (2013) proved that the biggest prime divisor of the $n$th term of a Lucas sequence of integers grows quicker than~$n$, answering famous questions of Erd\H os and Schinzel. In this note we obtain a fully explicit and, in a sense, uniform version of Stewart's result.

\end{abstract}

{\footnotesize

\tableofcontents

}

\section{Introduction}
\label{sintro}

For a non-zero algebraic number~$\gamma$, denote by ${\omega(\gamma)}$ the number of distinct primes~$\gerp$ of the field $\Q(\gamma)$ with the property ${\nu_\gerp(\gamma)\ne 0}$. We denote by  $\norm\gamma$ the $\Q$-norm:
${\norm\gamma=\norm_{\Q(\gamma)/\Q}(\gamma)}$. 

The following theorem was proved by Stewart in his seminal article~\cite{St13}.

\begin{theorem}
\label{thstew}
Let~$\gamma$ be a non-zero algebraic number, not a root of unity, satisfying the following: 
\begin{itemize}
\item
either ${\gamma\in \Q}$,

\item or ${[\Q(\gamma):\Q]=2}$ and ${\norm\gamma=\pm1}$. 
\end{itemize}
Then there exists $n_0$, depending only on $\omega(\gamma)$ and the field ${K=\Q(\gamma)}$, with the following property. For every ${n>n_0}$
   there exists a prime~$\gerp$ of~$K$, with the underlying rational prime~$p$, such that  
${\nu_\gerp(\gamma^n-1) \ge 1}$ and  
$$
p\ge n\exp\left(\frac1{104}\frac{\log n}{\log\log n}\right).
$$
\end{theorem}

This result answered famous questions posed by Erd\H os and Schinzel, see the introduction of~\cite{St13} for a historical account.

Note that Stewart's \cite[Theorem~1.1]{St13} is stated   in different terms, but what he actually proves is exactly Theorem~\ref{thstew} above. 

In this note we re-examine Stewart's argument with the following objectives:

\begin{description}
\item[(uniformity)]
we show that Stewart's $n_0$ depends only on~$\Q(\gamma)$, but not on $\omega(\gamma)$; 
in particular, if ${\gamma \in \Q }$ then~$n_0$ is an absolute constant; 

\item[(explicitness)]
we obtain a totally explicit expression for~$n_0$. 
\end{description}

We prove the following two theorems. 

\begin{theorem}
\label{thstrat}
Let~$\gamma$ be a non-zero rational  number, distinct from $\pm1$. 
Set ${n_0=\exp(10^6)}$. Then for every ${n\ge n_0}$  there is a prime number~$p$  such that 
${\nu_p(\gamma^n-1) \ge 1}$ and  
${p\ge n\exp\left(0.0005\frac{\log n}{\log\log n}\right)}$.   
\end{theorem}

\begin{theorem}
\label{thstquad}
Let~$\gamma$ be a non-zero algebraic number of degree~$2$, not a root of unity. We denote~$D_K$  the discriminant of the number field  ${K=\Q(\gamma)}$, and we set
$
{n_0=\exp\exp(\max\{10^9,3|D_K|\})}
$. 
Assume that ${\norm\gamma=\pm1}$. Then for every ${n\ge n_0}$   there exists a prime~$\gerp$ of~$K$, with the underlying rational prime~$p$, such that  
${\nu_\gerp(\gamma^n-1) \ge 1}$ and  
$$
p\ge n\exp\left(0.0002\frac{\log n}{\log\log n}\right).
$$
\end{theorem}

Our numerical constants $0.0005$  and $0.0002$ are worse than Stewart's $1/104$. On the other hand, our~$n_0$ do not depend on $\omega(\gamma)$. Our argument, being very close to Stewart's, allows one, in principle, to obtain $1/104$ (but, probably, not $1/102$), for the price of  increasing the numerical value of~$n_0$.

We deduce Theorems~\ref{thstrat} and~\ref{thstquad} from the following two theorems (again, essentially, due to Stewart, see \cite[Section~4]{St13}), which are of independent interest. We denote by $\height(\cdot)$ the absolute logarithmic height, see Section~\ref{sprelim}. We also denote ${\logast =\max\{\log ,1\}}$, and we denote by $\norm\gerp$  the absolute norm of the ideal~$\gerp$; that is, ${\norm\gerp=\#\calO_K/\gerp}$.

\begin{theorem}
\label{thordrat}
Let~$\gamma$ be a non-zero algebraic number of degree~$d$, not a root of unity.  
Set 
$
{p_0=\exp(80000 d (\logast d)^2)}
$.
Then for every prime~$\gerp$ of the field ${K=\Q(\gamma)}$ with ${\norm\gerp\ge p_0}$, and every positive integer~$n$   we have 
$$
\nu_\gerp(\gamma^n-1) \le \norm\gerp\exp\left(-0.002d^{-1}\frac{\log \norm\gerp}{\log\log \norm\gerp}\right)\height(\gamma)\logast n. 
$$
\end{theorem}

\begin{theorem}
\label{thordquad}
Let~$\gamma$ be as in Theorem~\ref{thstquad}; that is, a non-zero algebraic number of degree~$2$ and norm $\pm1$, but not a root of unity.  We again denote~$D_K$ the discriminant of the field ${K=\Q(\gamma)}$, and we set 
$
p_0=\exp\exp(\max\{10^8,2|D_K|\}). 
$
Then for every prime~$\gerp$ of~$K$ with underlying rational prime ${p\ge p_0}$,  and every positive integer~$n$   we have 
\begin{equation}
\label{eordquad}
\nu_\gerp(\gamma^n-1) \le p\exp\left(-0.001\frac{\log p}{\log\log p}\right)\height(\gamma)\logast n.
\end{equation}

\end{theorem}

\begin{remark}
\begin{enumerate}
\item
The principal tool in the proof of Theorem~\ref{thstrat} is Theorem~\ref{thordrat}, which holds not only for ${\gamma\in \Q}$, but for arbitrary algebraic~$\gamma$. One may wonder whether Theorem~\ref{thstrat} can be extended to this generality. One may expect the following statement: for~$n$ large enough, there exists a prime~$\gerp$ of the number field $\Q(\gamma)$ such that ${\nu_\gerp(\gamma^n-1)\ge 1}$ and 
$
{\norm\gerp \ge n\exp\left(c\frac{\log n}{\log\log n}\right)} 
$,
where~$c$ is a positive number not depending on~$n$. 

Unfortunately, the present argument does not seem to be capable of proving this. See Remark~\ref{recannot} for more details. 

\item
Our values of~$n_0$ and~$p_0$ are rather huge numerically.  In particular, in Theorems~\ref{thstquad} and~\ref{thordquad} our~$n_0$, respectively~$p_0$,  are double exponential in~$|D_K|$. Of course, this is quite unsatisfactory for practical purposes. Unfortunately, not much can be done here without involving substantially new ideas. The reason is that we have to use the numerical Prime Number Theorem from~\cite{BMOR18} (see Proposition~\ref{pcountsplit}). And using this theorem requires parameter~$x$ therein to be exponential in~$|D_K|$. Since in the subsequent proof of Theorem~\ref{thordquad} this~$x$ is set to be around $\log p$, this yields double exponential dependence in $|D_K|$. Note also that the original approach of Stewart leads to even triple exponential dependence, as explained in Section~\ref{smult}.

\end{enumerate} 
\end{remark}

We follow the main lines of Stewart's argument, with two changes. Uniformity in~$\gamma$  is achieved by using Lemmas~\ref{ltriv} and~\ref{lkumm}. Another deviation of Stewart's argument is of more technical nature and is explained in detail in Section~\ref{smult}. 

\paragraph{Plan of the article}
Our principal tool is Yu's~\cite{Yu13} bound for a $p$-adic logarithmic form. In Section~\ref{syu}, 
we present a simplified version of Yu's result adapted for our purposes. In Section~\ref{sordrat}, we prove Theorem~\ref{thordrat}. 

In Section~\ref{squad} and~\ref{smult}, we collect various facts about quadratic fields used in the proof of Theorem~\ref{thordquad}, which is proved afterwards in Section~\ref{sordquad}.  

In Section~\ref{scypr}, we recall basic facts about cyclotomic polynomials and primitive divisors, needed for the proofs of Theorems~\ref{thstrat} and~\ref{thstquad}. These latter are proved in the final Sections~\ref{sthrat} and~\ref{sthquad} respectively.


\section{Notation and preliminaries}
\label{sprelim}
Let~$K$ be a number field. We denote by  $D_K$ and $h_K$ the discriminant and the class number of~$K$. 
By a prime of~$K$ we mean a prime ideal of the ring of integers~$\calO_K$.  We denote by  $\F_\gerp$ the residue field ${\calO_K/\gerp}$, and ${\norm\gerp =\#\F_\gerp}$ the absolute norm of~$\gerp$. 

Let $\gera,\gerb$ be non-zero fractional ideals of~$K$. We call them \textit{involved} if there exists a $K$-prime~$\gerp$ such that 
${\nu_\gerp(\gera),\nu_\gerp(\gerb)\ne 0}$. 
If no such prime exists, then we call $\gera,\gerb$ \textit{disjoint} (so that ``not involved'' and ``disjoint'' are synonyms).  We call ${\alpha,\beta \in K^\times}$ \textit{involved}, resp. \textit{disjoint} if so are the principal ideals $(\alpha),(\beta)$.



We denote by  $\height(\alpha)$ the usual absolute logarithmic height of ${\alpha\in\bar\Q}$:
$$
\height(\alpha) = [K:\Q]^{-1}\sum_{v\in M_K}d_v\log^+|\alpha|_v,
$$
where ${\log^+ =\max\{\log , 0\}}$ and~$d_v$ denotes the local degree. Here~$K$ is an arbitrary number field containing~$\alpha$, and the places ${v\in M_K}$ are normalized to extend the standard places of~$\Q$; that is, ${|p|_v=p^{-1}}$ if ${v\mid p <\infty}$ and ${|x|_v=|x|}$ if ${v\mid\infty}$ and ${x\in \Q}$.

If~$K$ is a number field of degree~$d$ and ${\alpha\in K}$ then the following formula is an immediate consequence of the definition of the height: 
$$
\height(\alpha) =\frac{1}{d}\left(\sum_{\sigma:K\hookrightarrow\C}\log^+|\alpha^\sigma|+\sum_\gerp\max\{0,-\nu_\gerp(\alpha)\}\log\norm\gerp\right), 
$$
where the first sum runs over the complex embeddings of~$K$ and the second sum runs over the primes of~$K$.   If ${\alpha \ne 0}$ then ${\height(\alpha) =\height(\alpha^{-1})}$, and we obtain the formula 
\begin{equation}
\label{ehe}
\height(\alpha) =\frac{1}{d}\left(\sum_{\sigma:K\hookrightarrow\C}-\log^-|\alpha^{\sigma}|+\sum_\gerp\max\{0,\nu_\gerp(\alpha)\}\log\norm\gerp\right), 
\end{equation}
where ${\log^-=\min\{\log,0\}}$. 

Besides ${\log^+}$ and $\log^-$ we will also use  ${\logast =\max\{\log , 1\}}$. 

We use $O_1(\cdot)$ as the quantitative version of the familiar $O(\cdot)$ notation: ${A=O_1(B)}$ means ${|A|\le B}$. 

We will use the following estimates for the arithmetical functions $\omega(n)$, $\ph(n)$ and $\pi(x)$:
\begin{align}
\label{erome}
\omega(n) &\le 1.4 \frac{\log n}{\log\log n} && (n\ge 3), \\
\label{ersph} 
\ph(n) &\ge 0.5 \frac{n}{\log\log n} && (n\ge 10^{20}), \\
\label{erspi}
\frac{x}{\log x}\le \pi(x)& \le 1.3\frac{x}{\log x} && (x\ge 3). 
\end{align}
See \cite[Théorème~11]{Ro83}, \cite[Theorem~15]{RS62} and  \cite[page~69, Corollary~1]{RS62}.


\section{Logarithmic forms}
\label{syu}

In this section,~$K$ is a number field of degree~$d$, and~$\gerp$ is a prime of~$K$ with underlying rational prime ${p\ge 5}$. Note that we will have ${p\ge 5}$ in both Sections~\ref{sordrat} and~\ref{sordquad}, where  Theorem~\ref{thyu} will be applied: see~\eqref{epge5} and~\eqref{epge5bis}. Let~$u$ be such that~$K$ contains a primitive root of unity of order~$2^u$, but not of order $2^{u+1}$. We pick a primitive root of unity of order~$2^u$ and denote it~$\zeta$.  


Our principal tool will be the following result of Yu~\cite{Yu13}. Recall that ${\alpha \in K^\times}$ is called a $\gerp$-adic unit if ${\nu_\gerp(\alpha) =0}$.

\begin{theorem}
\label{thyu}
Let ${\alpha_1, \ldots, \alpha_k\in K^\times}$ be multiplicatively independent  $\gerp$-adic units Let~$\delta$ and~$\Omega$  be real numbers satisfying 
\begin{align*}
\delta&\le
\begin{cases}
\bigl[\F_\gerp^\times:\langle\bar\zeta, \bar\alpha_1, \ldots, \bar\alpha_k\rangle\bigr], & \text{if $[K(\alpha_1^{1/2}, \ldots,\alpha_k^{1/2}):K]=2^k$},\\
1, & \text{otherwise}, 
\end{cases}\\
\Omega&=\max\left\{\frac{\norm\gerp}{\delta}\left(\frac{k}{\log\norm\gerp}\right)^k,e^k\log\norm\gerp\right\},  
\end{align*}
where $\langle\bar\zeta, \bar\alpha_1, \ldots, \bar\alpha_k\rangle$ is the  subgroup of the multiplicative group~$\F_\gerp^\times$ generated by the images of ${\zeta, \alpha_1, \ldots, \alpha_k}$.

Furthermore, let ${b_1, \ldots, b_k}$ be rational integers, not all~$0$, and denote
$$
B=\max\{|b_1|, \ldots, |b_k|\}. 
$$
Then 
\begin{equation}
\label{eyu}
\nu_\gerp\bigl(\alpha_1^{b_1}\cdots \alpha_k^{b_k}-1\bigr)\le 10^5 d^{k+2}(\logast d)^3\cdot 30^k k^{5/2}(\logast k)\height(\alpha_1)\cdots\height(\alpha_k)
\Omega\logast B. 
\end{equation}
\end{theorem} 

\begin{proof}
This is a simplification (with slightly bigger numerical constants) of \cite[Lemma~3.1]{St13}, which, on its own, is a simplification of the main theorem of~\cite{Yu13}. 

Let us explain how we deduce~\eqref{eyu} from \cite[Lemma~3.1]{St13}. Note that our~$k$ corresponds to~$n$ in~\cite{St13}. 
We will repeatedly use the  observations of the following kind: for ${a\ge 0}$ and ${x,y\ge 1}$ we have 
${a+x+y\le (a+2)xy}$. 

Plugging the  estimates 
\begin{align*}
\max\{\logast B, (k+1)(5.4k+\log d)\}&\le \logast B \cdot 13k^2\logast d,\\
(k+1)^{1/2}&\le \sqrt2k^{1/2},\\
7e\frac{p-1}{p-2}&\le \frac{28}{3}e \qquad (\text{recall that ${p\ge 5}$}), \\
\log(e^4(k+1)d)&\le 8\logast d\logast k 
\end{align*}
into \cite[Lemma~3.1]{St13} (with~$n$ replaced by~$k$), we  bound the left-hand side of~\eqref{eyu} by 
\begin{equation}
\label{eugly}
376\left(\frac{28}{3}e\right)^k d^{k+2} \cdot 104\sqrt2k^{5/2}\logast k(\logast d)^3 \height(\alpha_1)\cdots\height(\alpha_k) \Omega\logast B. 
\end{equation}
This is clearly smaller that the right-hand side of~\eqref{eyu}. 
\end{proof}

\section{Proof of Theorem~\ref{thordrat}}
\label{sordrat}
The following lemma is totally trivial, but we state it here because it is our principal tool in making~$p_0$ independent of~$\gamma$.

\begin{lemma}
\label{ltriv}
Let~$K$ be a field, 
${\gamma_1, \ldots, \gamma_k\in K^\times}$  multiplicatively independent, and ${\gamma \in K^\times}$ not a root of unity. Then, after a suitable renumbering of ${\gamma_1, \ldots, \gamma_k}$, the numbers ${\gamma, \gamma_2, \ldots, \gamma_k}$ become multiplicatively independent. 
\end{lemma}

We will also need a lower bound for the height of an algebraic number.

\begin{lemma}
\label{lhe}
Let~$\gamma$ be an algebraic number of degree~$d$, not a root of unity. Then
\begin{align}
\label{edone}
\height(\gamma) &\ge \log 2 && \text{for $d=1$}, \\
\label{edtwo}
2\height(\gamma)&\ge \log\frac{1+\sqrt5}2 && \text{for $d=2$}, \\
\label{edall}
d\height(\gamma)&\ge \frac{1}{4(\logast d)^3} && \text {for any $d$}. 
\end{align}
\end{lemma}

\begin{proof}
Inequality~\eqref{edone} is trivial, and~\eqref{edtwo} is a famous result of Schinzel~\cite{Sc75} (see also~\cite{HS93} for a very simple proof). 
Inequality~\eqref{edall}, for sufficiently large~$d$, follows from the famous work of Dobrowolski~\cite{Do79}. To obtain it for all  
 ${d\ge 3}$, we invoke  Voutier's numerical adaptation~\cite{Vo96} of Dobrowolski's  result. In particular, \cite[Corollary~2]{Vo96} gives ${d\height(\gamma)\ge 2/(\log(3d))^3}$, which clearly implies~\eqref{edall} for ${d\ge 3}$. Finally, for  ${d\le 2}$ inequality~\eqref{edall} follows from~\eqref{edone} and~\eqref{edtwo}. 
\end{proof}

We can now start the proof of Theorem~\ref{thordrat}. To simplify notation, we denote ${P=\norm\gerp }$. We  will assume that 
\begin{equation}
\label{eassum}
\norm\gerp=P\ge p_0=\exp\bigl(80000 d (\logast d)^2 \bigr)
\end{equation}
throughout the proof. Since ${\norm\gerp \le p^d}$, we have
\begin{equation}
\label{epge5}
p\ge p_0^{1/d}\ge \exp\bigl(80000  (\logast d)^2 \bigr)\ge 5,  
\end{equation}
which is required to apply Theorem~\ref{thyu}. 



Let~${x}$   be specified later to satisfy
\begin{equation}
\label{eassux}
x\ge 200(\logast d)^2.  
\end{equation}
Denote ${k=\pi(x)}$.  Let ${\ell_1, \ldots, \ell_k}$  be the~$k$ primes not exceeding~$x$ numbered somehow, not necessarily in the increasing order. 
Since  ${\ell_1, \ldots, \ell_k}$ are multiplicatively independent and~$\gamma$ is not a root of unity, Lemma~\ref{ltriv} implies that, after renumbering, the numbers ${\gamma, \ell_2, \ldots, \ell_k}$ are multiplicatively independent.

We apply Theorem~\ref{thyu} with 
\begin{align*}
&\alpha_1=\frac{\gamma}{\ell_2\cdots\ell_k}; \qquad \alpha_i=\ell_i \quad (i=2, \ldots,k); \\ 
&b_i=n \quad (i=1,\ldots, k);  \qquad \delta=1. 
\end{align*}
Since ${\height(\ell_i)=\log\ell_i\le \log x}$, we obtain 
\begin{align*}
\nu_\gerp(\gamma^n-1) &=\nu_\gerp\bigl(\alpha_1^n\ell_2^n\cdots \ell_k^n-1\bigr)\\
&\le 
10^5d^{k+2}(\logast d)^3\cdot 30^{k}k^{5/2}(\logast k )\height(\alpha_1)(\log x)^{k-1}\Omega\logast n, 
\end{align*}
where 
${\Omega = \max\bigl\{P(k/\log P)^k, e^k \log P\bigr\}}$. 
We will see  later, when we specify~$x$, that 
\begin{equation}
\label{easstwo}
P(k/\log P)^k \ge  e^k\log P,
\end{equation}
and so we have ${\Omega = P(k/\log P)^k}$.

Using Lemma~\ref{lhe}, we estimate 
$$
\height(\alpha_1)\le \height(\gamma)+(k-1)\log x \le 4d(\logast d)^3\height(\gamma)k\log x . 
$$
Hence 
\begin{equation*}
\nu_\gerp(\gamma^n-1)\le 4\cdot10^5d^3(\logast d)^6k^{7/2}(\logast k) P\left(\frac{30dk\log x}{\log P}\right)^k\height(\gamma)\logast n .
\end{equation*}
We want to simplify this estimate.

It follows from~\eqref{eassux} that  ${k\ge \pi(200)= 46}$, which easily implies that
\begin{equation*}
4\cdot10^5k^{7/2}(\logast k)\le 2^k.
\end{equation*}
Also, using~\eqref{erspi}, we obtain 
\begin{equation*}
k\ge \frac{200(\logast d)^2}{\log\bigl(200(\logast d)^2\bigr)},  
\end{equation*} 
which implies that ${d^3(\logast d)^6\le 2^k}$. Indeed, this is obvious when ${d=1,2}$. When ${d\ge 3}$, it is sufficient to prove that
\[
\frac{200(\log d)^2}{\log(200)+2\log\log d}\ge \frac{3\log d+6\log\log d}{\log 2}.
\]
This is true since $200\log 2(\log d)^2\ge 9\log d(\log 200+2\log d)$ holds when $d\ge 3$.

Finally, again using~\eqref{erspi}, we estimate  
${k\log x\le  1.3 x}$. 
Since 
$$
2\cdot 2\cdot 1.3 \cdot 30 <160,
$$ 
this implies the estimate 
\begin{equation}
\label{easier}
\nu_\gerp(\gamma^n-1)\le  P\left(\frac{160dx}{\log P}\right)^k\height(\gamma)\logast n .
\end{equation}

It is the time to specify~$x$. We set 
${x=(1/400d)\log P}$, which gives 
\begin{equation}
\label{e23}
\nu_\gerp(\gamma^n-1)\le P\cdot 0.4^k\height(\gamma)\logast n. 
\end{equation}
Note that~\eqref{eassux} is satisfied with our choice of~$x$, because of~\eqref{eassum}. 

Now we are almost done. 
Once again using~\eqref{erspi}, we obtain
\begin{equation*}
k \ge \frac{x}{\log x}\ge \frac{1}{400d}\frac{\log P}{\log\log P}. 
\end{equation*}
Substituting this to~\eqref{e23}, we obtain  
\begin{equation*}
\nu_\gerp(\gamma^n-1)\le  P\exp\left(-0.002d^{-1}\frac{\log P}{\log\log P}\right)\height(\gamma)\logast n, 
\end{equation*}
as wanted.

We are left with checking that assumption~\eqref{easstwo} holds true with our choice of~$x$. It suffices to show that ${P\ge (e\log P)^{k+1}}$. As we have seen above, ${k\ge 46}$, and we use~\eqref{erspi} to obtain
$$
k+1 \le \frac{47}{46}\pi(x) \le 1.4  \frac{x}{\log x} \le \frac{1.4}{400}\frac{\log P}{\log((1/400)\log P)}. 
$$
Since ${P\ge e^{80000}}$, we have 
$$
\frac{1}{400}\log P\ge (\log P)^{0.4},\qquad e\log P \le (\log P)^{1.1 }.
$$
It follows that 
$$
k+1 \le 0.009\frac{\log P}{\log\log P}, \qquad (e\log P)^{k+1}\le P^{0.01}\le P. 
$$
This completes the proof of the theorem.


\section{Quadratic  fields}
\label{squad}
We need to recall some  facts about quadratic fields. In this section, unless otherwise stated,~$K$ denotes a quadratic field. We denote by~$D_K$ and~$h_K$ the discriminant and the class number of~$K$ respectively. If~$K$ is a real quadratic field then we denote by  $\eta_K$ the fundamental unit~$\eta$ satisfying ${\eta>1}$. It will be convenient to set ${\eta_K=1}$ for imaginary~$K$. We denote by~$\sigma$ the non-trivial Galois morphism of~$K$ over~$\Q$. 
Note that, when~$K$ is real, we have
\begin{equation}
\label{eleta}
\eta_K \ge \frac{1+\sqrt5}{2}. 
\end{equation}
We set  ${\mu=\#\calO_K^\times/2}$; in other words, 
$$
\mu=
\begin{cases}
3, & \text{if $K=\Q(\sqrt{-3})$}, \\
2, & \text{if $K=\Q(i)$}, \\
1, & \text{in all other cases}. 
\end{cases}
$$

\begin{proposition}
\begin{enumerate}
\item
Let~$K$ be an imaginary quadratic field. 
Then 
\begin{equation}
\label{ehleim}
h_K\le \mu\pi^{-1} |D_K|^{1/2}(2+\log|D_K|). 
\end{equation}

\item
Let~$K$ be a real quadratic field. Then 
\begin{equation}
\label{ehlere}
h_K\log\eta_K \le \pi^{-1} D_K^{1/2}(2+\log D_K). 
\end{equation}

\item
For any quadratic field,~$K$, we have 
\begin{align}
\label{ehleun}
h_K&\le 3|D_K|^{1/2}\log|D_K|, \\
\label{eeta}
\log\eta_K &\le |D_K|^{1/2}\log |D_K|. 
\end{align}
\end{enumerate}
\end{proposition}

\begin{proof}
Estimates~\eqref{ehleim} and~\eqref{ehlere} are well-known; see, for instance,  Theorems~10.1 and~14.3 in \cite[Chapter 12]{Hu82}. Estimate~\eqref{ehleun} follows, in the imaginary case, from ${\mu\le 3}$ and ${|D_K|\ge 3}$, and in the real case from~\eqref{eleta} and  
${|D_K|\ge 5}$. Finally,~\eqref{eeta} is trivial in the imaginary case, and in the real case it follows from ${h_K\ge 1}$ and ${D_K\ge 5}$. 
\end{proof}

Denote by  $\pi_s(x,K)$ the counting function of rational primes that split in~$K$. 

\begin{proposition}
\label{pcountsplit}
For
\begin{equation}
\label{exge}
x\ge \max\{10^{10},e^{|D_K|}\}
\end{equation}
we have
\begin{equation}
\label{ephkast}
\pi_s(x,K) \ge \frac{1}{2}\frac{x}{\log x}-\frac{\ph(|D_K|)}{320}\frac{x}{(\log x)^2}. 
\end{equation}
\end{proposition}

\begin{proof}
We denote ${D=D_K}$. 
An odd rational prime~$p$ splits in~$K$ if and only if ${(D/p)=1}$. Primes satisfying this condition belong to one of ${\ph(|D|)/2}$ residue classes $\bmod |D|$. If ${a\bmod |D|}$ is one such class, then for~$x$ satisfying~\eqref{exge} we have
$$
\pi(x;|D|,a) \ge \frac{1}{\ph(|D|)}\Li(x)-\frac{1}{160}\frac{x}{(\log x)^2}, 
$$
see Theorem~1.3 in Bennett et al.~\cite{BMOR18}. As usual, we denote by  ${\pi(x;m,a)}$ the counting function for primes in the congruence class ${a\bmod m}$.

Note that ${\Li(x)>x/\log x}$ for ${x\ge 7}$, because the function   
$$
f(x)=\Li(x)-x/\log x
$$ 
satisfies ${f'(x)=1/(\log x)^2>0}$ and ${f(7)=0.114\ldots>0}$. It follows that 
$$
\pi(x;|D|,a) \ge \frac{1}{\ph(|D|)}\frac{x}{\log x}-\frac{1}{160}\frac{x}{(\log x)^2}.  
$$
Summing up over the ${\ph(|D|)/2}$ residue classes $a\bmod |D|$, we obtain~\eqref{ephkast}.   
\end{proof}

\section{Multiplicatively independent elements}
\label{smult}

We retain the notation and conventions of Section~\ref{squad}.

Stewart's argument in the quadratic case \cite[Section~4]{St13} requires producing in~$K$ many multiplicatively independent elements of norm~$1$ and controllable height. 
Stewart uses for this purpose prime numbers~$p$ with the following properties:
\begin{itemize}
\item
$p$ splits in~$K$, and 
\item
the $K$-primes above~$p$ are principal. 
\end{itemize}
We call them \textit{Stewart primes} in the sequel.

Let $(\pi)$ be a principal $K$-prime above a Stewart prime~$p$. If~$K$ is imaginary then  ${|\pi|=|\pi^\sigma|=p^{1/2}}$. If~$K$ is real then, multiplying~$\pi$ by a suitable power of the fundamental unit~$\eta_K$, we may assume that 
${(p/\eta_K)^{1/2}\le |\pi|,|\pi^\sigma|\le (p\eta_K)^{1/2}}$. 

Stewart associates to~$p$  the algebraic number  ${\theta_p= \pi/\pi^\sigma}$.  For this $\theta_p$ we have ${\norm\theta_p=1}$, and
$$
\height(\theta_p) =\frac12\log p+O_1\left(\frac12\log\eta_K\right); 
$$ 
recall that $O_1(\cdot)$ is quantitative version of $O(\cdot)$, see Section~\ref{sprelim}. Clearly, numbers~$\theta_p$ corresponding to distinct Stewart primes~$p$ are multiplicatively independent.

Using the Class Field Theory and the Tchebotarev Density Theorem, one can show that the relative density of Stewart primes in the set of all primes  is $(2h_K)^{-1}$. Moreover, using recent explicit versions of the Tchebotarev Density Theorem, as in \cite{AK19,KW22,Wi13}, one can give a totally explicit lower estimate for the counting function of Stewart primes.

Unfortunately, following this path, we end up with a rather huge value for the constant $p_0$ in Theorem~\ref{thordquad}, triple exponential in the discriminant of~$K$. For instance,  Theorem~5 of~\cite{KW22} applies for ${x\ge x_1:=\exp(|D_L|^{12})}$, where, in our case,~$L$ is the Hilbert Class Field of~$K$. We have ${|D_L|=|D_K|^{h_K}}$, which would lead to a double exponential value for~$x_1$. And $p_0$, as it is clear from the proof of Theorem~\ref{thordquad}, is exponential in~$x_1$, leading to the triple exp dependence of~$p_0$ in $|D_K|$.  For this reason, we do not pursue this approach in the present article.

Instead of Stewart primes, which are quite sparse, we  use all (sufficiently large) split primes, which have relative density $1/2$. More precisely,  
denote by $S(K)$ the set of   rational primes~$p$  which split in~$K$ and satisfy ${p\ge |D_K|^{1/2}}$. 
To every ${p\in S(K)}$  
we want to associate a certain element ${\theta_p\in K}$. We do it as follows. 

Given ${p\in S(K)}$,  let~$\gerp$ be a $K$-prime above~$p$. Recall that every ideal class contains an integral ideal~$\gera$ such that ${\norm\gera< |D_K|^{1/2}}$. We take such~$\gera$ in the class of $\gerp^{-1}$, so that ${\gerp\gera}$ is a principal ideal. Let~$\alpha$ be a generator of $\gerp\gera$. Then ${|\alpha\alpha^\sigma|=\norm(\gerp\gera)}$. Note also that, since ${\norm\gera< |D_K|^{1/2}}$, the number~$\alpha$ is not involved with any prime from the set $S(K)$ other than~$p$ itself.  

If~$K$ is imaginary then ${|\alpha|=|\alpha^\sigma|=\norm(\gerp\gera)^{1/2}}$. If~$K$ is real then, multiplying~$\alpha$ be a suitable power of~$\eta_K$, we may assume that 
\begin{equation}\label{absofalpha}
\norm(\gerp\gera)^{1/2}\eta_K^{-1/2} \le |\alpha|,|\alpha^\sigma|\le \norm(\gerp\gera)^{1/2}\eta_K^{1/2} . 
\end{equation}
Now we set 
$$
\theta_p=\alpha/\alpha^\sigma. 
$$

\begin{proposition}
\label{pthetas}
\begin{enumerate}
\item
\label{inone}
For every ${p\in S(K)}$ we have ${\norm\theta_p=1}$ and 
$$
\height(\theta_p) = \frac12\log p+ O_1\left(\frac14\log|D_K|+\frac12\log\eta_K\right).  
$$

\item
\label{izesi}
In particular, if ${p\ge \exp(100|D_K|^{1/2}\log|D_K|)}$ then 
\begin{equation}
\label{ezesi}
\height(\theta_p) \le 0.51\log p. 
\end{equation}

\item
\label{idisj}
Each~$\theta_p$ is involved with~$p$, but disjoint from any other prime exceeding ${ |D_K|^{1/2}}$. In particular, it is disjoint from any prime belonging to the set $S(K)$. If~$\gerp$ is a $K$-prime over~$p$, then ${\nu_\gerp(\theta_p)=\pm1}$.

\item
\label{imind}
If ${p_1, \ldots, p_k}$ are distinct elements of $S(K)$ then ${\theta_{p_1}, \ldots, \theta_{p_k}}$ are multiplicatively independent. Moreover,  
${\bigl[K\bigl(\sqrt{\theta_{p_1}}, \ldots, \sqrt{\theta_{p_k}}\bigr):K\bigr]=2^k}$.


\end{enumerate}
\end{proposition}

\begin{proof}
Items~\ref{inone} and~\ref{idisj} follow from~\eqref{absofalpha}, the definition of~$\theta_p$ and the upper bound  ${\norm \gera< |D_K|^{1/2}}$. To prove item~\ref{izesi}, it suffices to show that
$$
0.01\log p \ge \frac14\log|D_K|+\frac12\log\eta_K. 
$$
In view of~\eqref{eeta}, this would follow from
$$
\log p \ge 100\left(\frac14\log|D_K|+\frac12|D_K|^{1/2}\log |D_K|\right). 
$$ 
And this is a consequence of our assumption about~$p$.

To prove item~\ref{imind}, denote $L_i= K\bigl(\sqrt{\theta_{p_1}}, \ldots, \sqrt{\theta_{p_i}}\bigr)$ (with the convention ${L_0=K}$), and let~$\gerp_i$ be a prime of~$K$ above~$p_i$. Item~\ref{idisj} implies that~$\gerp_i$ ramifies in~$L_i$ but not in~$L_{i-1}$. Hence ${[L_i:L_{i-1}]=2}$, whence the result. 
\end{proof}

Item~\ref{izesi} of this proposition suggests to count the split primes~$p$  satisfying ${p\ge \exp(100|D_K|^{1/2}\log|D_K|)}$. The following is an immediate consequence of Proposition~\ref{pcountsplit}.

\begin{corollary}
\label{ccountbigsplit}
For ${x\ge \exp(\max\{10^7, |D_K|\})}$
we have
\begin{equation*}
\pi_s(x,K)-\pi_s\bigl(\exp(100|D_K|^{1/2}\log|D_K|),K\bigr) \ge 0.49\frac{x}{\log x}. 
\end{equation*}
\end{corollary}
\begin{proof}
Using 
Proposition \ref{pcountsplit},
\begin{align}
\pi_s(x,K)-&\pi_s\bigl(\exp(100|D_K|^{1/2}\log|D_K|),K\bigr)\nonumber\\
\label{exlogx}
&\ge   \frac{1}{2}\frac{x}{\log x}-\frac{\ph(|D_K|)}{320}\frac{x}{(\log x)^2}-
\exp(100|D_K|^{1/2}\log|D_K|).  
\end{align}
Let us estimate both the extra terms in the right-hand side of~\eqref{exlogx}. We have ${\log x \ge |D_K|}$, which implies that
$$
\frac{\ph(|D_K|)}{320}\frac{x}{(\log x)^2} \le \frac{1}{320}\frac{x}{\log x}. 
$$
Next, using ${\log x \ge |D_K|}$ and ${\log x\ge 10^7}$, we obtain 
$$
100|D_K|^{1/2}\log|D_K|\le 100(\log x)^{1/2}\log\log x <0.6\log x.
$$
Hence 
$$
\exp(100|D_K|^{1/2}\log|D_K|) < x^{0.6} <10^{-10} \frac x{\log x},
$$
where we again use the assumption ${\log x\ge 10^7}$. 

We conclude that the right-hand side of~\eqref{exlogx} exceeds 
$$
\left(\frac12-\frac{1}{320}-10^{-10}\right)\frac x{\log x}, 
$$
which is bigger than $0.49x/\log x$. 
\end{proof}


\section{Proof of Theorem~\ref{thordquad}}
\label{sordquad}
As in Section~\ref{sordrat}, we start from a simple lemma. 

\begin{lemma}
\label{lkumm}
Let~$K$ be a field of characteristic~$0$, let ${\gamma_1, \ldots, \gamma_k\in K^\times}$ be such that 
$$
\bigl[K\bigl(\sqrt{\gamma_1}, \ldots,\sqrt{\gamma_k}\bigr):K\bigr]=2^k, 
$$
and let ${\gamma\in K^\times}$ be not a square in~$K$. Then, after suitable renumbering ${\gamma_1, \ldots, \gamma_k}$, we have 
\begin{equation}
\label{ekumm}
\bigl[K\bigl(\sqrt{\gamma},\sqrt{\gamma_2}, \ldots,\sqrt{\gamma_k}\bigr):K\bigr]=2^k.  
\end{equation}

\end{lemma}

\begin{proof}
Let ${\bar\gamma, \bar\gamma_1, \ldots, \bar\gamma_k}$ be the images of ${\gamma,\gamma_1, \ldots, \gamma_k}$ in the group ${K^\times/(K^\times)^2}$. Viewing the latter as an $\F_2$-vector space, the vectors ${\bar\gamma_1, \ldots, \bar\gamma_k}$ are linearly independent and vector~$\bar\gamma$ is non-zero. Hence, after renumbering, vectors ${\bar\gamma, \bar\gamma_2, \ldots, \bar\gamma_k}$ become linearly independent. 
This yields~\eqref{ekumm} by Kummer's theory, as given, for instance, in \cite[Section~VI.8]{La02}. Indeed, Theorem~8.1 therein
implies that 
$$
[K\bigl(\sqrt{\gamma},\sqrt{\gamma_2}, \ldots,\sqrt{\gamma_k}\bigr):K\bigr]=[\Gamma:(K^\times)^2],  
$$
where~$\Gamma$ is 
the subgroup of~$K^\times$ generated by ${\gamma, \gamma_2, \ldots, \gamma_k}$ and $(K^\times)^2$. 
The quotient  ${\Gamma/(K^\times)^2}$ is isomorphic, as $\F_2$-vector space, to the space generated by ${\bar\gamma, \bar\gamma_2, \ldots, \bar\gamma_k}$. Hence 
${[\Gamma:(K^\times)^2]=2^k}$, and we are done. 
\end{proof}

Now we are ready to start the proof of Theorem~\ref{thordquad}. 
In this section, ${K=\Q(\gamma)}$ and~$\gerp$ is a prime of~$K$, whose underlying rational prime~$p$ satisfies 
\begin{equation}
\label{eassumpag}
p\ge p_0=\exp\exp\bigl(\max\{10^8, 2|D_K|\}\bigr).
\end{equation}
In particular, 
\begin{equation}
\label{epge5bis}
p\ge 5,
\end{equation}
which is required to apply Theorem~\ref{thyu}.

If ${\norm\gerp=p}$ then Theorem~\ref{thordquad} follows from the case ${d=2}$ of Theorem~\ref{thordrat}. Therefore we will assume that ${\norm\gerp=p^2}$. In particular, the residue field $\F_\gerp$ is the finite field ${\F_{p^2}}$.

Let~$x$ be a positive real number to be specified later to satisfy
\begin{equation}
\label{exlower}
x\ge \exp(\max\{10^7, |D_K|\}). 
\end{equation} 
The results of Section~\ref{smult}  imply the following. There exists a positive integer~$k$ and distinct prime numbers ${\ell_1, \ldots, \ell_k \in S(K)}$ such that 
\begin{align}
&k\ge 0.49\frac{x}{\log x}, \qquad \ell_i \le x, \nonumber\\ 
\label{ehethetali}
&\height(\theta_{\ell_i}) \le 0.51\log x \qquad (1\le i \le k). 
\end{align}
Note also that 
\begin{equation}
\label{ekupper}
k \le \pi(x) \le 1.3\frac{x}{\log x}, 
\end{equation}
see~\eqref{erspi}.

Next, let~$r$ be the biggest positive integer with the following property: there exists ${\theta\in K^\times}$ such that ${\theta^r=\gamma}$. This~$\theta$ is not a square in~$K$ by the definition of $r$, and Lemma~\ref{lkumm} implies that, after renumbering ${\ell_1, \ldots, \ell_k}$, we have 
$$
\bigl[K(\sqrt\theta, \sqrt{\theta_{\ell_2}}, \ldots, \sqrt{\theta_{\ell_k}}):K\bigr]=2^k. 
$$

Denote by~$G$ the subgroup of the multiplicative group ${\F_\gerp^\times=\F_{p^2}^\times}$, consisting of elements of norm $\pm1$:
$$
G=\{x \in \F_{p^2}^\times: \norm_{\F_{p^2}/\F_p}x=\pm1\}. 
$$ 
Since the norm map $\norm: \F_{p^2}^\times\to \F_p^\times$ is surjective (see, for instance, \cite[Theorem~2.28(ii)]{LN97}),  we have ${[\F_{p^2}^\times:G]=(p-1)/2}$.  
The $\F_\gerp$-images of ${\theta, \theta_{\ell_2}, \ldots, \theta_{\ell_k}}$ belong to~$G$. Hence we can use Theorem~\ref{thyu} with 
\begin{align*}
&\alpha_1=\frac{\theta}{\theta_{\ell_2}\cdots\theta_{\ell_k}}; \qquad \alpha_i=\theta_{\ell_i}\quad (i=2,\ldots,k);\\ 
& b_i=nr \quad (i=1,\ldots, k); \qquad \delta=\frac{p-1}{2}; \qquad d=2 .
\end{align*}
Note that ${\alpha_1, \ldots, \alpha_k}$ are $\gerp$-adic units, as required in Theorem~\ref{thyu}. Indeed, item~\ref{idisj} of Proposition~\ref{pthetas} implies that each~$\alpha_i$ is disjoint from any rational prime exceeding $|D_K|^{1/2}$, except perhaps ${\ell_1, \ldots, \ell_k}$. We have ${p\ne \ell_1, \ldots, \ell_k}$ (because~$p$ is inert in~$K$, and the primes $\ell_i$ split in~$K$), and ${p\ge |D_K|^{1/2}}$ by~\eqref{eassumpag}. Hence~$\alpha_i$ is disjoint from~$p$, that is, it is a $\gerp$-adic unit.

Using the upper bound~\eqref{ehethetali} for the heights of $\theta_{\ell_i}$'s, we obtain 
\begin{align}
\nu_\gerp(\gamma^n-1) &=\nu_\gerp\bigl(\alpha_1^{nr}\theta_{\ell_2}^{nr}\cdots \theta_{\ell_k}^{nr}-1\bigr) \nonumber \\
\label{efirst}
&\le 
10^6\cdot 60^{k}k^{5/2}(\logast k )\height(\alpha_1)(0.51\log x)^{k-1}\Omega\logast(nr), 
\end{align}
where 
$$
\Omega = \max\left\{\frac{2p^2}{p-1}\left(\frac{k}{2\log p}\right)^k,2e^k\log p\right\}. 
$$
We will see later that
\begin{equation}
\label{comega}
\Omega=\frac{2p^2}{p-1}\left(\frac{k}{2\log p}\right)^k
\end{equation} 
with our choice of~$x$.  
Using ${p\ge p_0}$, this implies that 
$$
\Omega \le 2.1p \left(\frac{k}{2\log p}\right)^k. 
$$
Next, we have 
${2\height(\theta) \ge\log((1+\sqrt5)/2)}$ 
by Lemma~\ref{lhe}. Using this, the definition of $\alpha_1$ and the upper bound~\eqref{ehethetali} for the height of the $\theta_{\ell_i}$'s, we estimate
\[
\height(\alpha_1)\leq \height(\theta)+0.51(k-1)\log x\leq 5\height(\theta)k\log x= \frac{5}{r}\height(\gamma)k\log x. 
\]
Also, a quick verification shows that 
${\logast (nr)\le r\logast n}$
for all possible choices of~$n$ and~$r$. 
Substituting all these estimates into~\eqref{efirst}, we obtain 
\[
\nu_\gerp(\gamma^n-1)\leq 10^8k^{7/2}(\logast k) p\left(\frac{15.3k\log x}{\log p}\right)^k\height(\gamma)\logast n.
\]
We want to simplify this estimate. It follows from~\eqref{exlower} that 
$$
k\ge  0.49\frac{x}{\log x} \ge \exp(10^6),
$$ 
which easily implies that 
${10^8k^{7/2}(\logast k) <1.1^k}$. 
Also, ${k\log x\leq 1.3x}$ by~\eqref{ekupper}. Since ${1.1\cdot 1.3\cdot 15.3 < 30}$, we obtain the estimate 
\[
\nu_\gerp(\gamma^n-1)\leq p\left(\frac{30x}{\log p}\right)^k\height(\gamma)\logast n .
\]
Now we set $x=300^{-1}\log p$. Then~\eqref{exlower} is satisfied, and we have
\[
\nu_\gerp(\gamma^n-1)\leq p\cdot 0.1^k\height(\gamma)\logast n.
\]
Since 
\[
k\geq 0.49\frac{x}{\log x}
>\frac{ 0.001\log p}{\log\log p}, 
\]
we obtain
\[
\nu_\gerp(\gamma^n-1)\le p\exp\left(-0.002\frac{\log p}{\log\log p}\right)\height(\gamma)\logast n, 
\]
which is even better than wanted. 

It remains to verify that~\eqref{comega} holds with our choice of~$x$. It suffices to prove that ${p>(2e\log p)^{k+1}}$. Using the lower estimates ${p\ge \exp\exp(10^8)}$ and ${x\ge \exp(10^7)}$ (see~\eqref{eassumpag},~\eqref{exlower}) together with the upper estimate~\eqref{ekupper}, we obtain
\[
k+1<\frac{1.4x}{\log x}=\frac{1.4}{300} \frac{\log p}{\log\log p-\log 300} < 0.01\frac{\log p}{\log\log p}
\]
and ${\log(2e\log p)<2\log\log p}$. 
It follows that ${(k+1)\log(2e\log p)<\log p}$. 
This completes the proof of the theorem.


\section{Cyclotomic polynomials and primitive divisors}
\label{scypr}

In this section, we collect some  results on cyclotomic polynomials and primitive divisors. We denote by  $\Phi_n(t)$ the cyclotomic polynomial of order~$n$. Recall that ${\deg\Phi_n=\ph(n)}$, the Euler totient. 

The following results go back to Schinzel~\cite{Sc74}, but in the present form they can be found in~\cite{BL21}. Recall (see Section~\ref{sprelim}) that ${A=O_1(B)}$ means ${|A|\le B}$.

\begin{proposition}
\label{prcycpol}
\begin{enumerate}
\item
\label{ihephin}
Let~$\gamma$ be an algebraic number. Then 
$$
\height(\Phi_n(\gamma)) = \ph(n)\height(\gamma) +O_1(2^{\omega(n)} \log (\pi n)). 
$$
\item
\label{iarch}
Let~$\gamma$ be a complex algebraic number of degree~$d$, non-zero and not a root of unity. Then 
\begin{equation}
\label{elogphi}
\log|\Phi_n(\gamma)| \ge -10^{14}d^5\height(\gamma)\cdot 2^{\omega(n)}\logast n. 
\end{equation}
\end{enumerate}
\end{proposition}

\begin{proof}
Item~\ref{ihephin} is \cite[Theorem~3.1]{BL21}. Item~\ref{iarch} follows from   \cite[Corollary~3.5]{BL21}, which gives the inequality 
$$
\log|\Phi_n(\gamma)| \ge -10^{12}d^3(\height(\gamma)+1)\cdot 2^{\omega(n)}\log(n+1). 
$$
We have clearly ${\log(n+1) \le 1.3\logast n}$. Also,   Lemma~\ref{lhe} implies that 
$$
d(\height(\gamma)+1) \le d\height(\gamma) (1+4d(\logast d )^3) <10d^3\height(\gamma). 
$$
This proves~\eqref{elogphi}.  
\end{proof}

Let~$K$ be a number field of degree~$d$ and ${\gamma\in K^\times}$ not a root of unity. We consider the sequence ${u_n=\gamma^n-1}$. We call a $K$-prime~$\gerp$ \textit{primitive divisor} of $u_n$ if 
$$
\nu_\gerp(u_n)\ge 1, \qquad \nu_\gerp(u_k)= 0 \quad (k=1,\ldots n-1). 
$$
Let us recall some basic properties of primitive divisors. 

\begin{proposition}
\label{pprim}
\begin{enumerate}
\item
\label{ipd}
Let~$\gerp$ be a primitive divisor of $u_n$. Then  ${\nu_\gerp(\Phi_n(\gamma) )\ge 1}$ and ${\norm\gerp\equiv1\bmod n}$; in particular, ${\norm\gerp \ge n+1}$.

\item
\label{idegtwo}
Let~$\gerp$ be a primitive divisor of $u_n$ and~$p$ the rational prime underlying~$\gerp$. If~$\gamma$  is of degree~$2$ and absolute norm~$1$, then  ${p\equiv\pm1\bmod n}$. 

\item
\label{inpd}
Assume that ${n\ge 2^{d+1}}$. Let~$\gerp$ be not a primitive divisor of~$u_n$. Then ${\nu_\gerp(\Phi_n(\gamma))\le \nu_\gerp(n)}$.
\end{enumerate}
\end{proposition}

\begin{proof}
Item~\ref{inpd} is Lemma~4 of Schinzel~\cite{Sc74}; see also \cite[Lemma~4.5]{BL21}. Items~\ref{ipd} and~\ref{idegtwo}   are well-known, but we include short proofs for the reader's convenience. 

To prove item~\ref{ipd}, note first of all that we must have ${\nu_\gerp(\gamma) = 0}$, because ${\nu_\gerp(\gamma^n-1)>0}$. Furthermore,
$$
\nu_\gerp(\gamma^n-1)=\sum_{m\mid n}\nu_\gerp(\Phi_m(\gamma)), 
$$
where each summand is non-negative because ${\nu_\gerp(\gamma) = 0}$. 
Since~$\gerp$ is a primitive divisor of~$u_n$, we must have ${\nu_\gerp(\Phi_m(\gamma))=0}$ for every ${m<n}$. It follows that  ${\nu_\gerp(\Phi_n(\gamma)=\nu_\gerp(\gamma^n-1)\ge 1}$. 

Let~$\bar\gamma$ be the image of~$\gamma$ in ${\F_\gerp=\calO_K/\gerp}$. Then saying that~$\gerp$ is a primitive divisor of ${\gamma^n-1}$ is equivalent to saying that~$n$ is the order of~$\bar\gamma$ in the multiplicative group $\F_\gerp^\times$. In particular,~$n$ must divide ${\norm\gerp-1}$, the order of this group. This complete the proof of item~\ref{ipd}.

In item~\ref{idegtwo}, if  ${\norm\gerp=p}$ then the result follows from item~\ref{ipd} (and we do not need the assumption ${\norm\gamma=1}$).  Now assume that ${\norm\gerp=p^2}$. The subgroup 
$$
\{x \in \F_{\gerp}^\times: \norm_{\F_{\gerp}/\F_p}x=1\}
$$
is of order ${p+1}$, because the norm map is surjective \cite[Theorem~2.28(ii)]{LN97}. Since~$\bar\gamma$ belongs to this subgroup, we must have ${n\mid (p+1)}$. This proves item~\ref{idegtwo}.  
\end{proof}

\section{Proof of Theorem~\ref{thstrat}}
\label{sthrat}

We set ${n_0=\exp(10^6)}$ and we assume that ${n\ge n_0}$ in the sequel.

Let~$P$ be the biggest prime number~$p$ with the property ${\nu_p(\Phi_n(\gamma))\ge 1}$. 
We want to show that 
\begin{equation}
\label{eplow}
P\ge n\exp\left(0.0005\frac{\log n}{\log\log n}\right).
\end{equation}
We will deduce this from Theorem~\ref{thordrat}, used  with ${d=1}$, and the properties of cyclotomic polynomials and primitive divisors collected in Section~\ref{scypr}.

We apply equation~\eqref{ehe} with ${\alpha=\Phi_n(\gamma)}$. Here ${d=1}$, and we obtain the following:
\begin{equation}
\label{ehephi}
\height\bigl(\Phi_n(\gamma)\bigr)= -\log^-|\Phi_n(\gamma)|+ \sum_{p}\max\bigl\{0, \nu_p\bigl(\Phi_n(\gamma)\bigr)\bigr\}\log p.
\end{equation}
We estimate the first term in~\eqref{ehephi} using item~\ref{iarch} of Proposition~\ref{prcycpol}:
\begin{equation}
\label{elogphin}
 -\log^-|\Phi_n(\gamma)|\le 10^{14}\height(\gamma)\cdot 2^{\omega(n)}\log n.
\end{equation}
Next, let us  call~$p$ \textit{primitive}  if it is a primitive divisor of ${\gamma^n-1}$, as defined in Section~\ref{scypr}, and \textit{non-primitive} otherwise. We split the  sum in~\eqref{ehephi} into two sums:
$$
\sum_{p}\max\bigl\{0, \nu_p\bigl(\Phi_n(\gamma)\bigr)\bigr\}\log p=\sum_{\genfrac{}{}{0pt}{}{\text{$p$ primi-}}{\text{tive}}}+\sum_{\genfrac{}{}{0pt}{}{\text{$p$ non-}}{\text{primitive}}}
=\Sigma_{\text{p}}+\Sigma_{\text{np}}. 
$$
We estimate $\Sigma_{\text{np}}$ using item~\ref{inpd} of Proposition~\ref{pprim}: 
\begin{equation*}
\Sigma_{\text{np}}\le \sum_p\nu_p(n)\log p = \log n. 
\end{equation*}
Thus, 
$$
\height\bigl(\Phi_n(\gamma)\bigr)\le  10^{14}\height(\gamma)\cdot 2^{\omega(n)}\log n+\log n+\Sigma_{\text{p}}. 
$$
On the other hand, item~\ref{ihephin} of Proposition~\ref{prcycpol} implies the lower bound 
$$
\height\bigl(\Phi_n(\gamma)\bigr) \ge \ph(n)\height(\gamma) -2^{\omega(n)} \log (\pi n). 
$$
Combining the two bounds, we obtain the following lower estimate for $\Sigma_{\text{p}}$: 
\begin{equation}
\label{elspprem}
\Sigma_{\text{p}}  \ge \ph(n)\height(\gamma) -2^{\omega(n)} \log (\pi n) - 10^{14}\height(\gamma)\cdot 2^{\omega(n)}\log n-\log n .  
\end{equation}
Inequalities~\eqref{erome},~\eqref{ersph} and our assumption ${n\ge \exp(10^6)}$  imply that   
the right-hand side of~\eqref{elspprem} is bounded from below by $0.9\ph(n)\height(\gamma)$. Thus, we obtain the lower estimate 
\begin{equation}
\label{elsf}
\Sigma_{\text{p}}\ge 0.9\ph(n)\height(\gamma). 
\end{equation}

Now let us  bound~$\Sigma_{\text{p}}$ from above.  Recall  that primitive~$p$ satisfy ${p\equiv 1\bmod n}$. In particular, 
${p\ge n+1>n_0}$. Since our~$n_0$ is bigger than  the $p_0$ from Theorem~\ref{thordrat}, the latter applies, and we obtain, for primitive~$p$, the estimate
\begin{align*}
\nu_p\bigl(\Phi_n(\gamma)\bigr)=\nu_p(\gamma^n-1)\le p\exp\left(-0.002\frac{\log p}{\log\log p}\right)\height(\gamma)\log n.
\end{align*}
Since ${p>n>e^e}$, we have 
$$
\frac{\log p}{\log\log p} \ge \frac{\log n}{\log\log n}. 
$$
Hence
$$
\nu_p\bigl(\Phi_n(\gamma)\bigr)\le P\exp\left(-0.002\frac{\log n}{\log\log n}\right)\height(\gamma)\log n.  
$$
It follows that 
\begin{align*}
\Sigma_{\text{p}} &\le \sum_{\genfrac{}{}{0pt}{}{p\equiv 1 \bmod n}{p\le P}} \max\bigl\{0, \nu_p\bigl(\Phi_n(\gamma)\bigr)\bigr\}\log p \\
&\le \pi(P; n, 1) \height(\gamma)P\exp\left(-0.002\frac{\log n}{\log\log n}\right)\log n\log P ,
\end{align*}
where, as usual, $\pi(x;m,a)$ counts primes ${p\le x}$ satisfying ${p\equiv a \bmod m}$. 
To estimate $\pi(P; n, 1)$, Stewart uses the Brun-Titchmarsh inequality. However, just the trivial estimate ${\pi(P; n, 1) \le P/n}$ would suffice. We obtain 
\begin{equation}
\label{eusf}
\Sigma_{\text{p}} \le 2\height(\gamma)P^2\log P\exp\left(-0.002\frac{\log n}{\log\log n}\right) \frac{\log n}{n}. 
\end{equation}

Thus, we have a lower bound~\eqref{elsf} and an upper bound~\eqref{eusf} for $\Sigma_{\text{p}}$. Combining the two, we obtain 
$$
P^2\log P \ge 0.4\frac{n\ph(n)}{\log n} \exp\left(0.002\frac{\log n}{\log\log n}\right).
$$
We may assume that ${P<n^2}$, since otherwise there is nothing to prove. Using this assumption and~\eqref{ersph}, we obtain \begin{align*}
2P^2\log n &\ge P^2\log P \\
&\ge 0.4\frac{n\ph(n)}{\log n} \exp\left(0.002\frac{\log n}{\log\log n}\right)\\
&\ge 0.2\frac{n^2}{\log n\log\log n}\exp \left(
0.002\frac{\log n}{\log\log n}
\right).
\end{align*}
This can be re-written as  \[
P\ge \sqrt{0.1}\frac{n}{\log n\sqrt{\log\log n}}\exp \left(
0.001\frac{\log n}{\log\log n}\right). 
\]
Since ${n\ge \exp(10^6)}$, we must have $$ \sqrt{0.1}\frac{n}{\log n\sqrt{\log\log n}}\exp \left(
0.001\frac{\log n}{\log\log n}\right)\ge n\exp \left(
0.0005\frac{\log n}{\log\log n}
\right).$$
Hence~\eqref{eplow} is proved.

\begin{remark}
\label{recannot}
As it is already indicated in the introduction, Theorem~\ref{thstrat} holds not only for ${\gamma\in \Q}$, but for arbitrary algebraic~$\gamma$, and one may wonder whether Theorem~\ref{thstrat} can be extended to this generality, like: for~$n$ large enough, there exists a prime~$\gerp$ of the number field $\Q(\gamma)$ such that 
$$
\nu_\gerp(\gamma^n-1)\ge 1, \qquad 
\norm\gerp \ge n\exp\left(c\frac{\log n}{\log\log n}\right), 
$$
where~$c$ is a positive number not depending on~$n$. 

Unfortunately, the present argument does not seem to be capable of proving this. The reason is that, when ${\gamma\notin\Q}$, there is no good bound for the number of~$\gerp$ satisfying ${\norm\gerp\equiv 1\bmod n}$. For instance, if~$\gamma$ is of degree~$2$, we have to count rational primes satisfying ${p^2\equiv 1\bmod n}$. Since the ring $\Z/n\Z$ may have as much as $2^{\omega(n)}$ square roots of unity, we cannot obtain, without involving extra ideas, an upper bound sharper than $2^{\omega(n)}$ for the number of such primes. And, since $\omega(n)$ can be of magnitude as big as $\log n/\log\log n$, this would destroy the tiny gain ${\exp\left(-0.002d^{-1}\frac{\log \norm\gerp}{\log\log \norm\gerp}\right) }$ obtained in Theorem~\ref{thordrat}. 

In the case ${d=2}$ this difficulty is overcome in~\cite{Ho22}, using ideas from the previous article~\cite{BHL21}. However, the case ${d\ge 3}$ remains  open. 
\end{remark}

\section{Proof of Theorem~\ref{thstquad}}
\label{sthquad}

We follow the proof of Theorem~\ref{thstrat} with appropriate modification. In particular, we set
$
{n_0=\exp\exp(\max\{10^9,3|D_K|\})}
$
and we assume that ${n\ge n_0}$ throughout the proof. 

Let $P$ be the biggest element of the set \[\{ p:  \text{$p$ is a rational prime lying below a prime~$\gerp$ of $K$, with $\nu_\gerp(\Phi_n(\gamma))\geq 1$}\}.\]
We want to show that \begin{equation}\label{conclusion1.3}
P\ge n\exp\left(0.0002\frac{\log n}{\log\log n}\right).
\end{equation}

We apply equation~\eqref{ehe} with ${\alpha=\Phi_n(\gamma)}$. Here ${d=2}$, and we obtain 
\begin{equation}
\label{ehephi2}
2\height\bigl(\Phi_n(\gamma)\bigr)=-\log^-|\Phi_n(\gamma)|-\log^-|\Phi_n(\gamma^\sigma)|+ \sum_{\gerp}\max\bigl\{0, \nu_\gerp\bigl(\Phi_n(\gamma)\bigr)\bigr\}\log\norm\gerp,
\end{equation}
where~$\sigma$ is the non-trivial morphism of $\Q(\gamma)/\Q$. 

We use item~\ref{iarch} of Proposition~\ref{prcycpol} to estimate the first term of \eqref{ehephi2}:\begin{equation}
\label{elogphin2}
-\log^-|\Phi_n(\gamma)| -\log^-|\Phi_n(\gamma^\sigma)|\le 2^6\cdot 10^{14}\height(\gamma)\cdot 2^{\omega(n)}\log n.
\end{equation}
We split the sum in \eqref{ehephi2} into two parts:\[
 \sum_{\gerp}\max\bigl\{0, \nu_\gerp\bigl(\Phi_n(\gamma)\bigr)\bigr\}\log\norm\gerp=\sum_{\genfrac{}{}{0pt}{}{\text{$\gerp$ primi-}}{\text{tive}}}+\sum_{\genfrac{}{}{0pt}{}{\text{$\gerp$ non-}}{\text{primitive}}}
=\Sigma_{\text{p}}+\Sigma_{\text{np}}. 
\]
By item~\ref{inpd} of Proposition~\ref{pprim}, we can bound the non-primitive part,\[
\Sigma_{\text{np}}\le \sum_\gerp\nu_\gerp(n)\log \norm\gerp \le 2\log n. 
\]
Thus\begin{equation}\label{hupperb}
\height\bigl(\Phi_n(\gamma)\bigr)\le   10^{16}\height(\gamma)\cdot 2^{\omega(n)}\log n+\Sigma_{\text{p}}/2+\log n.
\end{equation}
On the other hand, by item~\ref{ihephin} of Proposition~\ref{prcycpol},\begin{equation}\label{hlowerb}
\height\bigl(\Phi_n(\gamma)\bigr)\ge\varphi(n)\height(\gamma)-2^{\omega(n)}\log(\pi n)
\end{equation}
Combining \eqref{hupperb} and \eqref{hlowerb}, we have\begin{equation}\label{sumplowerb}
\Sigma_{\text{p}}/2\ge\varphi(n)\height(\gamma)-2^{\omega(n)}\log(\pi n)-10^{16}\height(\gamma)2^{\omega(n)}\log n-\log n
\end{equation}
Inequalities~\eqref{erome},~\eqref{ersph} and our assumption ${n\ge n_0\ge \exp\exp(10^9)}$  imply that   
the right-hand side of~\eqref{sumplowerb} is bounded from below by $0.9\ph(n)\height(\gamma)$. Thus, we obtain the lower estimate 
\begin{equation}
\label{elsf2}
\Sigma_{\text{p}}\ge 1.8\ph(n)\height(\gamma). 
\end{equation}

Now let us  bound~$\Sigma_{\text{p}}$ from above. By item \ref{ipd} of Proposition \ref{pprim}, a primitive divisor $\gerp$ of $\gamma^n-1$ satisfies $\norm\gerp\equiv 1\bmod n$. In paticular $\norm\gerp\ge n+1$ and thus the underlying rational prime $p$ is bigger than $\sqrt{n_0}=\exp(\exp(\max\{10^9, 3|D_K|\})/2)$, which is bigger than the $p_0$ in Theorem~\ref{thordquad}. So we obtain, for primitive $\gerp$ with underlying prime $p$,
\begin{align*}
\nu_\gerp\bigl(\Phi_n(\gamma)\bigr)=\nu_\gerp(\gamma^n-1) \le p\exp\left(-0.001\frac{\log p}{\log\log p}\right)\height(\gamma)\log n.
\end{align*} 
Since ${p^2\ge \norm\gerp >n>e^e}$, we have 
$$
\frac{\log p}{\log\log p} \ge \frac12 \frac{\log n}{\log\log n}. 
$$
Hence
$$
\nu_\gerp\bigl(\Phi_n(\gamma)\bigr)\le \height(\gamma)P\exp\left(-0.0005\frac{\log n}{\log\log n}\right)\log n.  
$$
Using this and items~\ref{ipd},~\ref{idegtwo} of Proposition~\ref{pprim}, we obtain 
\begin{align*}
\Sigma_{\text{p}} &\le \sum_{\genfrac{}{}{0pt}{}{\norm\gerp\equiv 1 \bmod n}{p\le P}} \max\bigl\{0, \nu_\gerp\bigl(\Phi_n(\gamma)\bigr)\bigr\}\log \norm\gerp \\
&\le 2\bigl(\pi(P; n, 1)+\pi(P; n, -1)\bigr) \height(\gamma)P\exp\left(-0.0005\frac{\log n}{\log\log n}\right)\log n\log P .
\end{align*}
As in Section~\ref{sthrat}, we estimate trivially ${\pi(P; n, 1)+\pi(P; n, -1)\le 2P/n}$. 
We obtain
\begin{equation}
\label{eusf2}
\Sigma_{\text{p}} \le 8\height(\gamma) P^2\log P\exp\left(-0.0005\frac{\log n}{\log\log n}\right)\frac{\log n}{n}.
\end{equation}
Combining the lower bound~\eqref{elsf2} the upper bound~\eqref{eusf2}, we obtain 
$$
P^2\log P \ge 0.1\frac{n\ph(n)}{\log n} \exp\left(0.0005\frac{\log n}{\log\log n}\right).
$$
Using again ${n\ge \exp\exp(10^9)}$ we obtain~\eqref{conclusion1.3}, arguing as in the end of the proof of Theorem~\ref{thstrat} in Section~\ref{sordrat}.

\paragraph{Acknowledgments}
Yuri Bilu and Sanoli Gun acknowledge support of the SPARC Project P445
``Arithmetical aspects of the Fourier coefficients of modular forms''. Yuri Bilu was also supported by the ANR project JINVARIANT. Haojie Hong was supported by the China Scholarship Council grant CSC202008310189.

The authors thank Keith Conrad, Florian Luca,  Kevin O'Bryant and Fabien Pazuki  for useful discussions.  We especially thank the anonymous referee for careful reading of the manuscript and many suggestions, that helped us to correct mistakes and improve the presentation.

{\footnotesize

\label{pagbib}
\bibliographystyle{amsplain}
\bibliography{stew_k}

}

\paragraph{Yuri Bilu \& Haojie Hong:} Institut de Mathématiques de Bordeaux, Université de Bordeaux \& CNRS, Talence, France

\paragraph{Sanoli Gun:} The Institute of Mathematical Sciences, 
Taramani,
Chennai, 
Tamil Nadu, India

\end{document}